\documentclass[12pt, a4paper]{article}

\usepackage{amsmath, amssymb, amsthm, amsrefs}
\usepackage{caption}
\usepackage{enumitem}
\usepackage{graphicx}
\usepackage{geometry, color}

\numberwithin{equation}{section}

\theoremstyle{plain}
\newtheorem{thm}{Theorem}[section]

\newtheorem{lem}[thm]{Lemma}
\newtheorem{prop}[thm]{Proposition}
\newtheorem{exam}[thm]{Example}

\DeclareMathOperator{\diver}{div}
\DeclareMathOperator{\tr}{tr}

\def\ga{\alpha}
 
\def\gd{\delta}
\def\gD{\Delta}
\def\ge{\eta}
\def\gg{\gamma}
\def\gG{\Gamma} 
\def\gk{\kappa} 
\def\gl{\lambda}

\def\gO{\Omega} 
\def\gr{\rho}                 
\def\gs{\sigma}

\def\gt{\tau} 
\def\gth{\theta}

\def\ep{\varepsilon}     
 
\def\vth{\vartheta}

\def\cd{\cdot}
\def\lds{\ldots}

\def\tim{\times}
\def\otim{\otimes}

\def\wed{\wedge}

\def\sb{\subset}
\def\sm{\setminus}

\def\pl{\partial}
\def\inft{\infty}
\def\emp{\emptyset}

\def\ol{\overline}
\def\ul{\underline}

\def\til{\tilde}
\def\wtil{\widetilde}

\def\fr{\frac}
\def\cfr{\cfrac}

\def\R{\mathbb R}

\def\cD{\mathcal D}
\def\cS{\mathcal S}
\def\cI{\mathcal I}

\title{Asymptotic analysis for Hamilton-Jacobi equations with large drift term}

\author{Taiga Kumagai}

\date{}

\begin{document}

\maketitle

%%%%%%%%%%%%%%%%%%%% Abstract, key words, 2010 MSC %%%%%%%%%%%%%%%%%%%%
\begin{abstract}
We investigate the asymptotic behavior of solutions of Hamilton-Jacobi equations
with large drift term in an open subset of two-dimensional Euclidean space.
When the drift is given by $\ep^{-1} (H_{x_2}, -H_{x_1})$ of a Hamiltonian $H$, with $\ep > 0$,
we establish the convergence, as $\ep \to 0+$, of solutions of the Hamilton-Jacobi equations and
identify the limit of the solutions as the solution of systems of ordinary differential equations on a graph.
This result generalizes the previous one obtained by the author to the case where
the Hamiltonian $H$ admits a degenerate critical point and, as a consequence,
the graph may have segments more than four at a node.

\textit{Key Words and Phrases.} Singular perturbation, Hamilton-Jacobi equations, Large drift, Graphs.

2010 \textit{Mathematics Subject Classification Numbers.} 35B40, 49L25.
\end{abstract}

%%%%%%%%%%%%%%%%%%%% Introduction %%%%%%%%%%%%%%%%%%%%
\section{Introduction}

We consider the boundary value problem for the Hamilton-Jacobi equation
\begin{align} \label{epHJ} 
\begin{cases} 
\gl u^\ep - \cfr{b \cd Du^\ep}{\ep} + G(x, Du^\ep) = 0 \ \ \ &\text{ in } \gO, \tag{$\mathrm{HJ}^\ep$} \\ 
u^\ep = g^\ep \ \ \ &\text{ on } \pl \gO,
\end{cases}
\end{align}
and investigate the asymptotic behavior, as $\ep \to 0+$, of the solution $u^\ep$ to \eqref{epHJ}.

In the above and henceforth, $\ep$ is a small positive parameter, $\gl$ is a positive constant,
$\gO$ is an open subset of $\R^2$ with boundary $\pl \gO$,
$G : \ol \gO \tim \R^2 \to \R$ and $g^\ep : \pl \gO \to \R$ are given functions,
$b : \R^2 \to \R^2$ is a given vector field, $u^\ep : \ol \gO \to \R$ is the unknown, and
$Du^\ep$ and $p \cdot q$ denote, respectively, the gradient of $u^\ep$ and the Euclidean inner product of $p, q \in \R^2$.
We give the vector field $b$ as a Hamilton vector field, that is, for a given Hamiltonian $H : \R^2 \to \R^2$,
\begin{equation*}
b = (H_{x_2}, -H_{x_1}),
\end{equation*}
where the subscript $x_i$ indicates the differentiation with respect to the variable $x_i$.

We are interested in the Hamiltonian flow with one degree of freedom
\begin{equation} \label{HS}
\dot X (t) = b(X(t)) \ \ \ \text{ and } \ \ \ X(0) = x \in \R^2, \tag{HS}
\end{equation}
and with its perturbed system 
\begin{equation} \label{control}
\dot X^\ep (t) = b(X^\ep(t)) + \ep \ga (t) \ \ \ \text{ and } \ \ \ X^\ep(0) = x \in \R^2, 
\end{equation}
where $\ga \in L^\inft (\R ; \R^2)$.  
Rescaling the time from $t$ to $t/\ep$ in \eqref{control}, we obtain
\begin{equation} \label{state}
\dot X^\ep (t) = \fr{1}{\ep} \, b(X^\ep(t)) + \ga (t) \ \ \ \text{ and } \ \ \ X^\ep (0) = x \in \R^2.
\end{equation}

The first equation of \eqref{epHJ} is the dynamic programming equation for the optimal control problem.
As is well-known, the viscosity solution $u^\ep$ of \eqref{epHJ} is identified with
the value function of the optimal control problem,
where the state equation, the discount factor, the pay-off at the exit time, and the running cost are given, respectively,
by \eqref{state}, $\gl$, $g^\ep$, and the function $L$, defined by
\begin{equation*}
L(x, \xi) = \sup_{p \in \R^2} \{ -\xi \cd p - G(x, p) \} \ \ \ \text{ for } (x, \xi) \in \ol \gO \tim \R^2.
\end{equation*}    
Thus, the investigation of the asymptotic behavior, as $\ep \to 0+$, of
the solutions of \eqref{epHJ} may be regarded in a broad sense as analyzing
the behavior, as $\ep \to 0+$, of the solutions of \eqref{state}, with ``optimal'' $\ga$.

In a spirit similar to the above,
but with a stochastic perturbation in place of a ``perturbation by optimal control'',  
Freidlin and Wentzell in \cite{FW} has initiated the study of a stochastic perturbation for \eqref{HS} and
established a convergence theorem for the solutions of
the linear second-order uniformly elliptic partial differential equation (pde, for short) 
\begin{equation}\label{model}
-\fr{b \cd Du^\ep}{\ep}- \gD u^\ep = f(x), 
\end{equation}
with a continuous function $f$ on $\ol \gO$.
Here, a similarity of the elliptic pde above to the pde in \eqref{epHJ} is that 
$\gl > 0$ and $G(x,Du^\ep)$ in \eqref{epHJ} correspond, respectively,
to $\gl=0$ and $-\Delta u^\ep-f$ in \eqref{model}. 
Regarding the stochastic perturbation,
Ishii and Souganidis in \cite{IS} has established a convergence theorem similar to that in \cite{FW}, 
by a pure pde-techniques, which covers a fairly general linear second-order degenerate elliptic pdes.

Motivated by the developments (\cite{FW, IS}) in stochastic perturbations for \eqref{HS}, 
the author in \cite{K} has recently established a convergence result for Hamilton-Jacobi equations \eqref{epHJ}
by using viscosity solution techniques such as the perturbed test function method and
representations of solutions as value functions in optimal control. 
A typical Hamiltonian $H$ studied in \cite{FW, IS, K} is given by 
\begin{equation*} \label{2-well}
H(x_1, x_2) = x_1^2 + \fr{1}{2} (x_2^2 - 1)^2 - \fr{1}{2}, 
\end{equation*}
whose graph has the shape of the so-called double-well potential, 
and it has three non-degenerate critical points at $(0, -1)$, $(0, 0)$ and $(0, 1)$.  
In this case, the limiting functions in the convergence results
either in stochastic perturbation or in perturbation by optimal control 
are characterized by systems of odes on graphs $\gG$ with one node and three edges, where,
roughly speaking, one of the edges corresponds to one of the potential wells,
another to the other potential well and the last to a finite tube above the potential wells.

Our main contribution in this article is to prove a convergence theorem
when the Hamiltonian $H$ has  degenerate critical points.  
The result is stated as Theorem \ref{main} below,
where the graph on which the limit function is defined has one node and
arbitrarily many edges depending on the Hamiltonian.

A simple example of such Hamiltonians is given by 
\begin{equation*} \label{exam}
H(x_1, x_2) = (x_1^2+x_2^2)^2-3x_1^2x_2+x_2^3.
\end{equation*}
We emphasize that most work on stochastic perturbation of Hamiltonian flows has studied
the case where Hamiltonian $H$ has only non-degenerate critical points.

Now, we mention that related problems have been considered 
in the context of Hamilton-Jacobi equations on networks or graphs.
In particular, the convergence results of approximated solutions
by fattening networks or graphs were established
in \cite{AT} for Hamilton-Jacobi equations in optimal control and
in \cite{LS} for non-convex Hamilton-Jacobi equations.

An interesting point of the result in \cite{K} is that
we have to treat a non-coercive Hamiltonian (that is, $-b(x)\cdot p/\ep + G(x, p)$) in \eqref{epHJ},
while very few authors have studied Hamilton-Jacobi equation with non-coercive Hamiltonian on networks or graphs. 
This difficulty due to lack of coercivity is resolved by taking the advantage that
the Hamiltonian $-b(x) \cd p/\ep + G(x, p)$ in \eqref{epHJ} is coercive in the direction orthogonal to $b(x)$.
See \cite{K} for details.

The graphs considered in Hamilton-Jacobi equations on networks or graphs,
in general, have many number of segments at a node.
However, in perturbation analysis of Hamiltonian flows as discussed above, 
when Hamiltonian $H$ has only non-degenerate critical points,
the number of segments at a node of graph $\Gamma$ is at most ``four'' (see, for example, \cite{F})
because, in this case, $H$ can be represented only by
\begin{equation} \label{saddle}
H(x_1, x_2) = x_1^2 - x_2^2
\end{equation}
in a neighborhood of a saddle point, which is corresponding to a node on $\Gamma$.

The argument in \cite{K} depends heavily on the formula \eqref{saddle}, which 
allows us to use an explicit formula of the solution of \eqref{HS} in a neighborhood of a saddle point. 
This is a main crucial point to establish our convergence theorem
since it is impossible to find a convenient explicit formula of solutions of \eqref{HS}
for general Hamiltonian $H$ with degenerate critical points. 
The idea to overcome the difficulty above is to use geometric integral formulas
for some quantities of the flow \eqref{HS} instead of solving \eqref{HS} explicitly.

This paper is organized as follows.
In the next section, we first present the assumptions on Hamiltonian $H$, 
typical examples of $H$, and the domain $\Omega$.
After these, we describe a basic existence and uniqueness proposition for \eqref{epHJ}
as well as the assumptions on the function $G$ throughout this paper, and
we finally state the main result (see Theorem \ref{main}).
In Section 3 divided into two parts,
we study some properties of functions in the odes in the limiting problem and subsolutions to the odes. 
Section 4 is devoted to the proof of Theorem \ref{main} along the argument in \cite{K}.
In Section 5, we present a sufficient condition, similar to that in \cite{K},
on the boundary data for the odes on the graph for which (G5) and (G6) hold.

Finally, we give a few of our notations.

%%%%%%%%%%%%%%%%%%%% Notation %%%%%%%%%%%%%%%%%%%%
\subsection*{Notation}

For $r > 0$, we denote by $B_r$ the open disc centered at the origin with radius $r$.
For $c, d \in \R$, we write $c \wed d = \min \{ c, d \}$.

%%%%%%%%%%%%%%%%%%%% Preliminaries and Main result %%%%%%%%%%%%%%%%%%%%
\section{Preliminaries and Main result}

%%%%%%%%%%%%%%%%%%%% The Hamiltonian %%%%%%%%%%%%%%%%%%%%
\subsection{The Hamiltonian}

Let $N \geq 3$.
We assume the following assumptions on the Hamiltonian $H$ throughout this paper.

\begin{itemize}
\item[(H1)] $H \in C^2 (\R^2)$ and $\lim_{|x| \to \inft} H(x) = \inft$.

\item[(H2)] $H$ has exactly $N$ critical points $z_i \in \R^2$, with $i \in \{ 0, \lds, N-1 \}$, and
                attains local minimum at $z_i$ and $i \in \{ 1, \lds, N-1 \}$.
\end{itemize}

Here and henceforth, we write
\begin{equation*}
\cI_0 = \{ 0, \lds ,N-1 \} \ \ \ \text{ and } \ \ \ \cI_1 = \{ 1, \ldots, N-1 \}.
\end{equation*}

For example, in the case where $N=4$, 
the graph of the Hamiltonian $H$ satisfying (H1) and (H2) is shaped like Fig. 2 below.

The number $N$ in (H2) coincides with
number of segments at a node of a graph arising in the limiting process.

To simplify the notation, we assume without loss of generality that
\begin{equation*}
z_0 = 0 := (0, 0) \ \ \ \text{ and } \ \ \ H(0) = 0.
\end{equation*}
We remark that, under assumptions (H1) and (H2),
in the case where $N \geq 4$, the origin is just a degenerate critical point of the Hamiltonian $H$, while,
in the case where $N = 3$, it may be a non-degenerate one.

\begin{itemize}
\item[(H3)] There exist constants $m \geq 0$ and $C > 0$, and 
		a neighborhood $V \sb \R^2$ of the origin such that, for any $i, j \in \{ 1, 2\}$,
                \begin{equation*}
		|H_{x_i x_j}(x)| \leq C|x|^m \ \ \ \text{ for all } x \in V.
                \end{equation*}
\end{itemize}

We note that assumption (H3) implies, by replacing $C > 0$ by a larger number if necessary, that
\begin{equation} \label{order-DH}
|DH(x)| \leq C|x|^{m+1} \ \ \ \text{ for all } x \in V,
\end{equation}
and
\begin{equation} \label{order-H}
|H(x)| \leq C|x|^{m+2} \ \ \ \text{ for all } x \in V.
\end{equation}

\begin{itemize}
\item[(H4)] There exist constants $n, c > 0$ such that $n < m+2$ and  
		\begin{equation*}
		c|x|^n \leq |DH(x)| \ \ \ \text{ for all } x \in V.
		\end{equation*}
\end{itemize}

Combining (H4) with \eqref{order-DH}, we see that $m+1\leq n$, and
with \eqref{order-H}, we get the relation
\begin{equation} \label{H-DH}
c_0 |H(x)|^\frac{n}{m+2} \leq |DH(x)| \ \ \ \text{ for all } x \in V
\end{equation}
for some $c_0 > 0$.

The following examples show that
conditions (H3) and (H4) are satisfied for wide range of Hamiltonians $H$.

\begin{figure}[t]
\begin{minipage}{0.5\hsize}
\centering
\input{N=3.tex}
\caption{$H$ $(N = 3)$}
\end{minipage}
\begin{minipage}{0.5\hsize}
\centering
\input{N=4.tex}
\caption{$H$ $(N = 4)$}
\end{minipage}
\end{figure}

\begin{exam} \label{H_3}
{\rm
Consider two Hamiltonians
\begin{equation*}
H_3 (x_1, x_2) = (x_1^2+ x_2^2)^2-x_1^2+x_2^2 \ \ \ \text{ and } \ \ \ \ H_3^o (x_1, x_2) = (x_1^2+x_2^2)^3-x_1^4+x_2^4.
\end{equation*}
It is obvious that $H_3$ and $H_3^o$ satisfy (H1).
By simple computations, we see that $H_3$ and $H_3^o$ satisfy (H3) and (H4) 
with, respectively, $(m, n)=(0, 1)$ and $(m, n)=(2, 3)$.
Both of the number of critical points are three, which consists of the origin; a saddle point, and,
respectively, $z_{1,2} = (\pm \sqrt{2}/2, 0)$ and $z_{1, 2} = (\pm \sqrt{6}/3, 0)$; local minimum points.
That is, (H2) holds with $N=3$.
The origin is a degenerate critical point of $H_3^o$, while it is a non-degenerate one of $H_3$.
The graphs of $H_3$ and $H_3^o$ are shaped like Fig. 1.
}
\end{exam}

\begin{exam} \label{H_4}
{\rm
Next, consider the Hamiltonian
\begin{equation*}
H_4 (x_1, x_2) = (x_1^2+x_2^2)^2-3x_1^2x_2+x_2^3.
\end{equation*}
It is easy to check that $H_4$ satisfies (H1)--(H4) with $(m, n)=(1, 2)$ and $N = 4$.
The critical points of $H_4$ are the origin; a degenerate saddle point, and
$z_1 = (3\sqrt{3}/8, 3/8)$, $z_2 = (-3\sqrt{3}/8, 3/8)$, and $z_3 = (0, -3/4)$; local minimum points.
The graph of $H_4$ is shaped like Fig. 2.
To understand the shape of $H_4$ well, 
we remark that $H_4$ can be represented in polar coordinates by
\begin{equation*}
\wtil H_4 (r, \gth)= r^4 - r^3 \sin 3\gth,
\end{equation*}
that is, the zero-level set of $H_4$ is the curve expressed by $r = \sin 3\gth$.
Indeed, if $x_1+ \, \mathrm{i} \, x_2 := re^{\, \mathrm{i} \, \gth}$, 
where $\mathrm{i}$ denotes the imaginary unit, then
\begin{equation*}
r^3 \sin 3\gth = r^3 \, \mathrm{Im} \, e^{3 \, \mathrm{i} \, \gth} = \mathrm{Im} \, r^3e^{3 \, \mathrm{i} \, \gth} 
                   = \mathrm{Im} \, (x_1+ \, \mathrm{i} \, x_2)^3 = 3x_1^2x_2-x_2^3.
\end{equation*}

}
\end{exam}

\begin{exam} \label{H_N}
{\rm
More generally, the Hamiltonian
\begin{equation*}
H_N (x_1, x_2) = (x_1^2+x_2^2)^{\fr{N}{2}} + \sum_{k=1}^{[N/2]} (-1)^k \binom{N-1}{k} x_1^{N-2k}x_2^{2k-1}
\end{equation*}
satisfies (H1)--(H4) with $(m, n)=(N-3, N-2)$ provided $N \geq 4$.
Here $[y]$ denotes the largest integer less than or equal to $y \in \R$.
Similarly to $H_4$ in Example \ref{H_4},
we see that the zero-level set of $H_N$ is the curve expressed by $r = \sin (N-1)\gth$
through the representation in polar coordinates by 
\begin{equation*}
\wtil H_N (r, \gth) = r^N - r^{N-1}\sin (N-1)\gth,
\end{equation*}
where
\begin{equation*}
r^{N-1} \sin (N-1)\gth = \sum_{k=1}^{[N/2]} (-1)^k \binom{N-1}{k} x_1^{N-2k}x_2^{2k-1}.
\end{equation*}
The critical points of $\wtil H_N$ are the origin and
\begin{equation*}
(r, \gth_i) = \left( \fr{N-1}{N}, \ \sin \fr{4i-3}{2} \, \pi \right) \ \ \ \text{ for } i \in \cI_1,
\end{equation*}
which are, respectively, corresponding to a saddle point and local minimum points of $H_N$.
} 
\end{exam}

%%%%%%%%%%%%%%%%%%%% The domain %%%%%%%%%%%%%%%%%%%%
\subsection{The domain}

Under assumptions (H1) and (H2), 
for any $h > 0$, the open set $\{ x \in \R^2 \mid H(x) < h \}$ is connected, and
the open set $\{ x \in \R^2 \mid H(x) < 0 \}$ consists of
$N-1$ connected components $D_1, \lds, D_{N-1}$ such that $z_i \in D_i$.

We choose the real numbers
\begin{equation*}
h_0 > 0 \ \ \ \text{ and } \ \ \ H(z_i) < h_i < 0 \ \ \ \text{ for } i \in \cI_1,
\end{equation*}
and set the intervals
\begin{equation*}
J_0 = (0, h_0) \ \ \ \text{ and } \ \ \ J_i = (h_i, 0) \ \ \ \text{ for } i \in \cI_1.
\end{equation*}

We put the open sets
\begin{equation*}
\gO_0 = \{ x \in \R^2 \mid H(x) \in J_0 \} \ \ \ \text{ and } \ \ \ 
\gO_i = \{ x \in D_i \mid H(x) \in J_i \} \ \ \ \text{ for } i \in \cI_1, 
\end{equation*}
and their ``outer" boundaries
\begin{equation*}
\pl_i \gO = \{ x \in \ol \gO_i \mid H(x) = h_i \} \ \ \ \text{ for } i \in \cI_0.
\end{equation*}

Finally, we introduce $\gO$ as the open connected set 
\begin{equation*}
\gO = \left( \bigcup_{i = 0}^{N-1} \gO_i \right) \cup \{ x \in \R^2 \mid H(x) = 0 \},
\end{equation*}
with the boundary
\begin{equation*}
\pl \gO = \bigcup_{i = 0}^{N-1} \pl_i \gO.
\end{equation*}

For example, the shapes of $\gO$ corresponding to $H$ in Figs. 1 and 2 are,
respectively, depicted in Figs. 3 and 4.

By (H1), the initial value problem \eqref{HS} admits
a unique global in time solution $X(t, x)$ such that
\begin{equation*} \label{C^1-X}
X, \dot X \in C^1 (\R \tim \R^2; \R^2).
\end{equation*}

As is well known, $H$ is a first integral for the system \eqref{HS}, that is,
\begin{equation*}
H(X(t, x)) = H(x) \ \ \ \text{ for all } (t, x) \in \R \tim \R^2. 
\end{equation*}

For $h \in \bar J_i$ and $i \in \cI_0$, 
we define the loops $c_i (h)$ by
\begin{equation*}
c_i (h) = \{ x \in \ol \gO_i \mid H(x) = h \}.
\end{equation*}
If we identify all points belonging to a loop $c_i (h)$, 
we obtain a graph $\gG$ consisting of $N$ segments parametrized by $J_0, \lds, J_{N-1}$.
For example, the graph $\gG$ corresponding to $\gO$ in Fig. 4 is shown in Fig. 5.

\begin{figure}[t]
\centering
\input{Omega-Gamma}
\caption{$\gG$ $(N = 4)$}
\end{figure}

It is not hard to check the following facts:
if $h \in J_i \cup \{ h_i \}$ and $i \in \cI_0$, then,
for any $\bar x_i \in c_i(h)$, the map $t \mapsto X(t, \bar x_i)$ is periodic and
\begin{equation*}
c_i (h) = \{ X(t, \bar x_i) \mid t \in \R \}.
\end{equation*}
If $i \in \cI_1$, then, for any $\bar x_i \in c_i(0) \sm \{ 0 \}$,
\begin{equation*}
c_i(0) = \{ 0 \} \cup \{ X(t, \bar x_i) \mid t \in \R \} \ \ \ \text{ and } \ \ \ \lim_{t \to \pm \inft} X(t, \bar x_i) = 0,
\end{equation*}
and
\begin{equation*}
c_i(0) = \pl D_i.
\end{equation*}
Moreover,
\begin{equation*}
c_0(0) = \{ x \in \R^2 \mid H(x) = 0 \} = \bigcup_{i \in \cI_1} c_i(0)
\end{equation*}
and
\begin{equation*}
c_0(0) = \pl D_0.
\end{equation*}

\begin{figure}[t]
\begin{minipage}{0.5\hsize}
\centering
\input{OmegaN=3.tex}
\caption{$\gO$ $(N = 3)$}
\end{minipage}
\begin{minipage}{0.5\hsize}
\centering
\input{OmegaN=4.tex}
\caption{$\gO$ $(N = 4)$}
\end{minipage}
\end{figure}

%%%%%%%%%%%%%%%%%%%% The Hamilton-Jacobi equation %%%%%%%%%%%%%%%%%%%%
\subsection{The Hamilton-Jacobi equation}

We put the following assumptions (G1)--(G5) on $G$ and $g^\ep$ throughout this paper.

\begin{itemize}
\item[(G1)] $G \in C(\ol \gO \tim \R^2)$.

\item[(G2)] There exists a modulus $m$ such that
                \begin{equation*}
                |G(x, p) - G(y, p)| \leq m(|x - y|(1 + |p|)) \ \ \ \text{ for all } x, y \in \ol \gO \text{ and } p \in \R^2.
                \end{equation*}

\item[(G3)] For each $x \in \ol \gO$, the function $p \mapsto G(x, p)$ is convex on $\R^2$.

\item[(G4)] $G$ is coercive, that is, 
               \begin{equation*}
               G(x, p) \to \inft \ \ \ \text{ uniformly for } x \in \ol \gO \text{ as } |p| \to \inft.
               \end{equation*}
\end{itemize}

Assumption (G2) is a standard requirement to $G$ that
the comparison principle should hold for \eqref{epHJ}.
Under assumptions (G1), (G3), and (G4), there exist $\nu, M > 0$ such that
\begin{equation} \label{coercivity}
G(x, p) \geq \nu |p| - M \ \ \ \text{ for  all } (x, p) \in \ol \gO \tim \R^2.
\end{equation}

The following condition has the same role as compatibility conditions described in \cite{L},
which are used to ensure the continuity up to boundary of the value functions in optimal control.
That is, it guarantees the continuity up to the boundary of the function $u^\ep$ of the form \eqref{value} below and, 
hence, gives us the uniqueness of viscosity solutions of \eqref{epHJ}.
In what follows, we write $X^\ep (t, x, \ga)$ for the solution to \eqref{state}.

\begin{itemize}  
\item[(G5)] There exists $\ep_0 \in (0, 1)$ such that 
		the family $\{ g^\ep \}_{\ep \in (0, \ep_0)} \sb C(\pl \gO)$ is uniformly bounded on $\pl \gO$ and
		that, for any $\ep \in (0, \ep_0)$,
		\begin{equation*}
                g^\ep (x) \leq \int_0^\vth L(X^\ep (t, x, \ga), \ga (t)) e^{-\gl t} \, dt + g^\ep (y) e^{-\gl \vth} 
                \end{equation*}
                for all $x, y \in \pl \gO$, $\vth \in [0, \inft)$, and $\ga \in L^\inft (\R; \R^2)$,
                where the conditions
                \begin{equation*}
                X^\ep (\vth, x, \ga) = y \ \ \ \text{ and } \ \ \ X^\ep (t, x, \ga) \in \ol \gO \ \ \ \text{ for all } t \in [0, \vth]
                \end{equation*}
                are satisfied, that is, $\vth$ is a visiting time at $y$ of the trajectory $\{ X^\ep (t, x, \ga) \}_{t \geq 0}$ constrained in $\ol \gO$.
\end{itemize}

We state here a basic existence and uniqueness proposition for \eqref{epHJ}.

\begin{prop} \label{solution}
For $\ep \in (0, \ep_0)$, we define the function $u^\ep : \ol \gO \to \R$ by
\begin{align} \label{value}
\begin{aligned}
u^\ep (x) = \inf \Big\{ \int_0^{\gt^\ep} L(X^\ep (t, x, \ga), \ga &(t))e^{-\gl t} \, dt \\
            &+ g^\ep (X^\ep (\gt^\ep, x, \ga))e^{-\gl \gt^\ep} \mid \ga \in L^\inft (\R; \R^2) \Big\},
\end{aligned}
\end{align}
where $\gt^\ep$ is a visiting time in $\pl \gO$ of $\{ X^\ep (t, x, \ga) \}_{t \geq 0}$ constrained in $\ol \gO$,
that is, $\gt^\ep$ is a nonnegative number such that
\begin{equation*}
X^\ep (\gt^\ep, x, \ga) \in \pl \gO \ \ \ \text{ and } \ \ \ X^\ep (t, x, \ga) \in \ol \gO \ \ \ \text{ for all } t \in [0, \gt^\ep].
\end{equation*} 
Then $u^\ep$ is the unique viscosity solution of \textrm{\eqref{epHJ}} and
continuous on $\ol \gO$, and satisfies $u^\ep = g^\ep$ on $\pl \gO$.
Furthermore the family $\{ u^\ep \}_{\ep \in (0, \ep_0)}$ is uniformly bounded on $\ol \gO$.
\end{prop}

Noting that \eqref{order-DH} implies, in particular, that $|b(x)| \leq |x|$ for all $x \in V$,
we can prove this proposition along the same lines as the proof of \cite[Proposition 2.3]{K}, so we skip it here.

Thanks to this proposition, we may define hereafter $u^\ep$ by \eqref{value}.
Since the family $\{ u^\ep \}_{\ep \in (0, \ep_0)}$ is uniformly bounded on $\ol \gO$,
the half relaxed-limits, as $\ep \to 0+$, of $u^\ep$
\begin{align*}
&v^+ (x) = \lim_{r \to 0+} \sup \{ u^\ep (y) \mid y \in B_r (x) \cap \ol \gO, \ \ep \in (0, r) \}, \\
&v^- (x) = \lim_{r \to 0+} \inf \{ u^\ep (y) \mid y \in B_r (x) \cap \ol \gO, \ \ep \in (0, r) \}
\end{align*}
are well-defined and bounded on $\ol \gO$.

If $v^+(x) \not= v^-(x)$ for some $x \in \pl \gO$,
a boundary layer happens in the limiting process of sending $\ep \to 0+$.
In order that any boundary layer does not occur, 
in addition to (G1)--(G5), we henceforth assume the following.

\begin{itemize}
\item[(G6)] There exist constants $d_i$ such that $v^\pm (x) = d_i$ for all $x \in \pl_i \gO$ and $i \in \cI_0$.
\end{itemize}
It is obvious that this leads to
\begin{equation*} \label{g-d_i}
\lim_{\gO \ni y \to x} v^\pm (y) = \lim_{\ep \to 0+} g^\ep (x) = d_i \ \ \ \text{ uniformly for } x \in \pl_i \gO \text{ for all } i \in \cI_0.
\end{equation*}

Our asymptotic analysis of \eqref{epHJ} is based on rather implicit (or ad hoc) assumptions (G5) and (G6),
which are indeed convenient for our arguments below.
However, it is not clear which $g^\varepsilon$ and $d_i$ satisfy (G5) and (G6). 
Thus, it is important to know when (G5) and (G6) hold.
In \cite{K}, for $N = 3$, the author gave a fairly general sufficient condition on the data $d_i$, for which (G5) and (G6) hold.
In Section 5, for more general $N$, we will present a similar condition to that in \cite{K}.

%%%%%%%%%%%%%%%%%%%% Main result %%%%%%%%%%%%%%%%%%%%
\subsection{Main result} \label{main result}

We introduce some notation which are needed to state our main result.

For $i \in \cI_0$ and $h \in \bar J_i$, let $L_i(h)$ denote the length of $c_i(h)$, that is,
\begin{equation} \label{L_i}
L_i (h) = \int_{c_i (h)} dl.
\end{equation} 
Here $dl$ denotes the line element.
Obviously, $L_i(h)$ are positive and bounded.

Recall that, if $h \in J_i \cup \{ h_i \}$ and $i \in \cI_1$,
then the map $t \mapsto X(t, \bar x_i)$ is periodic for any $\bar x_i \in c_i(h)$.
Note that the minimal periods are independent of choice of $\bar x_i \in c_i (h)$.
Hence, we can write $T_i (h)$ for the minimal period of the trajectory of the system \eqref{HS} on $c_i (h)$.
Noting that $|b(x)| = |DH(x)|$, the minimal period $T_i (h)$ has the form
\begin{equation} \label{T_i}
T_i (h) = \int_{c_i (h)} \frac{1}{|DH|} \, dl,
\end{equation}
which shows, in view of (H2), that
\begin{equation*}
0< T_i (h) < \inft \ \ \ \text{ and } \ \ \ \lim_{J_i \ni r \to 0} T_i (r) = \inft.
\end{equation*}

For $i \in \cI_0$,
define the function $\ol G_i : J_i \cup \{ h_i \} \tim \R \to \R$ by
\begin{equation*}
\ol G_i (h, q) = \cfrac{1}{T_i (h)} \int_{c_i (h)} \fr{G(x, qDH)}{|DH|} \, dl = \fr{1}{T_i(h)} \int_0^{T_i(h)} G(X(t, x), qDH(X(t, x))) \, dt,
\end{equation*}
where $x \in c_i(h)$ is fixed arbitrarily. 
We note here that the second formula above reveals that
$\ol G_i(h, q)$ is the mean value of the function $G(\cd, qDH(\cd))$
along the curve $X(t, x)$ on the loop $c_i(h)$.

We then state our main result.

\begin{thm} \label{main}
There exist viscosity solutions $u_i \in C(\bar J_i)$, with $i \in \cI_0$, of
\begin{equation} \label{limHJ}
\gl u + \ol G_i (h, u') = 0 \ \ \ \text{ in } J_i, \tag{$\mathrm{HJ}_i$} 
\end{equation} 
such that $u_i (h_i) = d_i$, $u_1 (0) = \lds = u_{N-1} (0)$, and, as $\ep \to 0+$,
\begin{equation*}
u^\ep \to u_i \circ H \ \ \ \text{ uniformly on } \ol \gO_i.
\end{equation*}
That is, if we define $\bar u \in C(\ol \gO)$ by
\begin{align*}
\bar u (x) = u_i \circ H (x) \ \ \ \text{ if } x \in \ol \gO_i, 
\end{align*}
then, as $\ep \to 0+$,
\begin{equation*}
u^\ep \to \bar u \ \ \ \text{ uniformly on } \ol \gO.
\end{equation*}
\end{thm}

We will give the proof of this theorem in Section 4.

%%%%%%%%%%%%%%%%%%%% The limiting problem %%%%%%%%%%%%%%%%%%%%
\section{The limiting problem}

%%%%%%%%%%%%%%%%%%%%  The minimal period and the length %%%%%%%%%%%%%%%%%%%%
\subsection{The minimal periods and the length}

In this subsection, we show an integrability of $T_i (h)$ and behavior of $c_i(h)$ near the origin
without any explicit formula of $X(t, x)$, which is a crucial difference from \cite{K}.

For this, we need the following lemma, 
which we refer to \cite[Lemma 1.1, Section 8]{FW}.

\begin{lem} \label{diff-H}
Let $i \in \cI_0$ and $r_1, r_2 \in J_i \cup \{ h_i \}$ be such that $r_1 < r_2$.
Set $D (r_1, r_2) = \{ x \in \gO_i \mid r_1 < H(x) < r_2 \}$.
If $f \in C^1(\ol{D(r_1, r_2)})$, then
\begin{equation*}
\int_{c_i (r_2)} f|DH| \, dl - \int_{c_i (r_1)} f|DH| \, dl  = \int_{D (r_1, r_2)} (Df \cd DH + f \gD H) \, dx,
\end{equation*}
and, moreover, for any $h \in J_i \cup \{ h_i \}$,
\begin{equation*}
\fr{d}{dh} \int_{c_i(h)} f|DH| \, dl = \int_{c_i (h)} \left( \fr{Df \cd DH}{|DH|} + f \fr{\gD H}{|DH|} \right) \, dl.
\end{equation*}
\end{lem}

\begin{lem} \label{C^1-L_iT_i}
For all $i \in \cI_0$, $L_i, T_i \in C^1 (J_i \cup \{ h_i \})$.
\end{lem}

\begin{proof}
Together with \eqref{L_i} and \eqref{T_i}, Lemma \ref{diff-H},
with $f = 1/|DH|$ and with $f = 1/|DH|^2$, yields respectively
\begin{equation*}
\fr{d}{dh} L_i (h) = \int_{c_i (h)} \left( -\fr{\tr DH \otim DH D^2H}{|DH|^4} + \fr{\gD H}{|DH|^2} \right) \, dl,
\end{equation*}
and
\begin{equation*}
\fr{d}{dh} T_i (h) = \int_{c_i (h)} \left( -\fr{2 \tr DH \otim DH D^2H}{|DH|^5} + \fr{\gD H}{|DH|^3} \right) \, dl,
\end{equation*}
where $D^2 H$ and $p \otim q$ denote the Hessian matrix of $H$ and 
the matrix $(p_iq_j)_{1 \leq i, j \leq 2}$ for $p, q \in \R^2$, respectively,
which imply that $L_i, T_i \in C^1(J_i \cup \{ h_i \})$.
\end{proof}

The behavior of $T_i (h)$ near $h = 0$ is the subject of

\begin{lem} \label{order-T_i}
For all $i \in \cI_0$, $T_i (h) = O(|h|^{-\fr{n}{m+2}})$ as $J_i \ni h \to 0$.
\end{lem}

Here, $m$ and $n$ are the constants from, respectively, (H3) and (H4).
We henceforth write $L(c)$ for the length of a given curve $c$.

\begin{proof}
Fix any $i \in \cI_0$ and $h \in J_i$.
Choose $\gk > 0$ so that $\ol B_\gk \sb \gO \cap V$.
By replacing $h$ if necessary so that $|h|$ is small enough,
we may assume that $c_i (h) \cap B_\gk \not= \emp$.

Set $\mu = \min_{c_i (h) \sm \ol B_\gk} |DH| > 0$.
By \eqref{H-DH}, we compute that
\begin{align*}
T_i (h) &= \int_{c_i (h) \cap B_\gk} \fr{1}{|DH|} \, dl + \int_{c_i (h) \sm \ol B_\gk} \fr{1}{|DH|} \, dl \\
         &\leq \int_{c_i (h) \cap B_\gk} \fr{1}{c_0 |H|^{\fr{n}{m+2}}} \, dl + \mu^{-1} L(c_i (h) \sm \ol B_\gk)
         \leq c_0^{-1} L_i (h)|h|^{-\fr{n}{m+2}} + \mu^{-1} L_i (h),
\end{align*}
from which we conclude that $T_i (h) = O(|h|^{-\fr{n}{m+2}})$ as $J_i \ni h \to 0$.
\end{proof}

We remark that since $n < m+2$, this lemma assures that
\begin{equation} \label{T_i-L^1}
T_i \in L^1(J_i) \ \ \ \text{ for all } i \in \cI_0.
\end{equation}
This integrability ensures the continuity of solutions of \eqref{limHJ} up to $h = 0$ (See Lemma \ref{extend}).

\begin{lem} \label{length}
Let $\gk \in (0, 1)$ be a constant such that $\ol B_\gk \sb \gO \cap V$.
Then, there exists a constant $K > 0$, independent of $h$, such that, for any $r \in (0, \gk)$,
\begin{equation*}
L(c_i (h) \cap B_r) \leq K r^{m-n+2} \ \ \ \text{ for all } h \in J_i \text{ and } i \in \cI_0.
\end{equation*}
\end{lem}

\begin{proof}
Fix any $r \in (0, \gk)$, $i \in \cI_0$, and $h \in J_i$.
If $c_i (h) \cap B_r = \emp$, then $L(c_i (h) \cap B_r) = 0$.
So, we assume henceforth that $c_i (h) \cap B_r \not= \emp$ and, 
for the time being, that $i = 1$.

Set $W = \{ x \in \gO_1 \mid H(x) < h \}$ and $U = W \cap B_r$.
Let $g = (g_1, g_2) \in C^1(\ol U)$.
Green's theorem yields
\begin{equation*}
\int_{\pl U} g \cd \nu \, dl = \int_U \diver g \, dx,
\end{equation*}
where $\nu$ is the outer normal vector on $\pl U$.
Note that $\pl U = (c_1 (h) \cap B_r) \cup (W \cap \pl B_r)$ and that 
$\nu = DH/|DH|$ on $c_1 (h) \cap B_r$ and $\nu = x/r$ on $W \cap \pl B_r$.

Now, setting $g = DH/|DH|$, we have
\begin{equation} \label{Green}
\int_{c_1 (h) \cap B_r} \, dl + \int_{W \cap \pl B_r} \fr{DH \cd x}{r|DH|} \, dl = \int_U \diver \left( \fr{DH}{|DH|} \right) \, dx.
\end{equation}

Here, we have
\begin{equation} \label{Int1}
\left| \int_{W \cap \pl B_r} \fr{DH \cd x}{r|DH|} \, dl \right| \leq \int_{W \cap \pl B_r} \fr{|DH||x|}{r|DH|} \, dl = \int_{W \cap \pl B_r} \, dl \leq 2\pi r.
\end{equation}

On the other hand, we compute that
\begin{equation*}
\diver \left( \fr{DH}{|DH|} \right) = \fr{1}{|DH|} \tr[(I - \ol{DH} \otim \ol{DH})D^2H],
\end{equation*}
where $I$ denotes the identity matrix of size two
and $\bar p$ is defined by
\begin{equation*}
\bar p = \fr{p}{|p|} \ \ \ \text{ for } p \in \R^2 \sm \{ 0 \}.
\end{equation*}
Since $|\ol{DH}|=1$, by (H3), we have 
\begin{equation*}
|\tr[(I - \ol{DH} \otim \ol{DH})D^2H]| \leq 2 \sum_{i, j=1}^2 |H_{x_ix_j}| \leq 8C|x|^m,
\end{equation*}
and, hence, by (H4),
\begin{equation*}
\left| \diver \left( \fr{DH}{|DH|} \right) \right| \leq \fr{8C|x|^m}{|DH|} \leq \fr{8C}{c} \, |x|^{m-n}.
\end{equation*}
Since $U \sb B_r$, from the inequality above, we get
\begin{equation} \label{Int2}
\left| \int_U \diver \left( \fr{DH}{|DH|} \right) \, dx \right| 
\leq \int_{B_r} \left| \diver \left( \fr{DH}{|DH|} \right) \right| \, dx = \fr{16C\pi}{c(m-n+2)} \, r^{m-n+2}.
\end{equation}

Now, \eqref{Green}--\eqref{Int2} together yield
\begin{equation*}
L(c_1 (h) \cap B_r) = \int_{c_1 (h) \cap B_r} \, dl \leq 2\pi \left( 1 + \frac{8C}{c(m-n+2)} \right) r^{m-n+2}.
\end{equation*}
Here, we have used the fact that
$0 < m-n+2 \leq 1$.
Arguments for the other $i$ parallel the above.
The proof is complete.
\end{proof}

%%%%%%%%%%%%%%%%%%%%  The Hamiltonians $\ol G_i$ in \eqref{limHJ} %%%%%%%%%%%%%%%%%%%%
\subsection{The Hamiltonians $\ol G_i$ in \eqref{limHJ}}

We state here some properties of the functions $\ol G_i$.

\begin{lem} 
For any $i \in \cI_0$, $\ol G_i \in C(J_i \cup \{ h_i \} \tim \R)$, and
\begin{equation} \label{loc-coercive}
\ol G_i (h, q) \geq \fr{\nu L_i (h)}{T_i (h)} |q| - M \ \ \ \text{ for all } (h, q) \in J_i \cup \{ h_i \} \tim \R, 
\end{equation} 
that is, $\ol G_i$ is locally coercive in the sense that,
for any compact interval $I$ of $J_i \cup \{ h_i \}$,
\begin{equation*}
\lim_{r \to \inft} \inf \{ \ol G_i (h, q) \mid h \in I, \ |q| \geq r \} = \inft.
\end{equation*}
\end{lem}

Here, $\nu$ and $M$ are the constants from \eqref{coercivity}.

\begin{proof}
Combining the definition of $\ol G_i$
with Lemma \ref{C^1-L_iT_i} yields the continuity of $\ol G_i$, and 
with \eqref{coercivity} yields \eqref{loc-coercive}.
\end{proof}

\begin{lem} \label{limmin-G_i}
We have 
\begin{equation*}
\lim_{J_i \ni h \to 0} \min_{q \in \R} \ol G_i (h, q) = G(0, 0) \ \ \ \text{ for all } i \in \cI_0.
\end{equation*}
\end{lem}

\begin{proof}
Fix any $\gg > 0$.
We begin by noting that, 
due to the continuity and the coercivity of $G$,
there exist $R, C_1 > 0$ such that
$G(x, p) \geq G(0, 0)$ for all $(x, p) \in \ol \gO \tim (\R^n \sm B_R)$ and
\begin{equation} \label{R-C_1}
|G(x, p) - G(0, 0)| \leq \gg + C_1 (|x|^n + |p|^{\fr{m-n+2}{n}} |DH(x)|^{1 - \fr{m-n+2}{n}})
\end{equation}
for all $(x, p) \in \ol \gO \tim \ol B_R$.
Combining this with (H4) yields
\begin{equation} \label{G-1}
G(x, p) \geq G(0, 0) - \gg - C_2(|DH(x)| + |p|^{\fr{m-n+2}{n}} |DH(x)|^{1 - \fr{m-n+2}{n}})
\end{equation}
for all $(x, p) \in \ol \gO \tim \R^2$, and 
\begin{equation} \label{G-2}
|G(x, 0) - G(0, 0)| \leq \gg + C_2|DH(x)| \ \ \ \text{ for all } x \in \ol \gO
\end{equation}
for some $C_2 > 0$.

Fix any $i \in \cI_0$ and $h \in J_i$.
Using \eqref{G-2}, we get
\begin{equation} \label{G-3}
|\ol G_i (h, 0)-G(0, 0)| \leq \gg + \fr{C_2L_i(h)}{T_i(h)}.
\end{equation}

Fix any $q \in \R$ and set 
\begin{equation*}
S = \{ x \in c_i (h) \mid |x| \leq (R/c|q|)^{\fr{1}{n}} \},
\end{equation*}
where $c > 0$ is the constant from (H4).
If $x \in c_i (h) \sm S$, then, by (H4),
\begin{equation*}
\frac{R}{c|q|} < |x|^n \leq \fr{|DH(x)|}{c},
\end{equation*}
that is, $|q||DH(x)| > R$. Hence, we have
\begin{equation} \label{G-4}
G(x, qDH(x)) \geq G(0, 0) \ \ \ \text{ for all } x \in c_i (h) \sm S.
\end{equation}

Choose $\gk \in (0, 1)$ and $\gd > 0$ so that $\ol B_\gk \sb \gO \cap V$
and that if $|h| \leq \delta$, then
\begin{equation*}
\left( \fr{R}{c \gg T_i (h)} \right)^{\fr{1}{n}} < \gk.
\end{equation*}
If $|h| < \gd$ and $|q| > \gg T_i(h)$, then $(R/c|q|)^{\fr{1}{n}} < \gk$, and, hence, 
by Lemma \ref{length}, there exists a constant $K > 0$, independent of $h$, such that
\begin{equation} \label{G-5}
L \left( c_i (h) \cup B_{(R/c|q|)^{\fr{1}{n}}} \right) \leq K \left( \fr{R}{c|q|} \right)^{\fr{m-n+2}{n}}.
\end{equation}

Using \eqref{G-1} if $|q| \leq \gg T_i(h)$ and, otherwise, combining
\eqref{G-1} with \eqref{G-4} and \eqref{G-5},
we can compute, as the proof of \cite[Lemma 6.5]{K}, that
\begin{equation*}
\liminf_{J_i \ni h \to 0} \min_{q \in \R} \ol G_i (h, q) \geq G(0, 0),
\end{equation*}
while \eqref{G-3} yields
\begin{equation} \label{G-6}
\lim_{J_i \ni h \to 0} \ol G_i (h, 0) = G(0, 0).
\end{equation}
These two together complete the proof.
\end{proof}

%%%%%%%%%%%%%%%%%%%%  Properties of viscosity subsolutions of \eqref{limHJ} %%%%%%%%%%%%%%%%%%%%
\subsection{Properties of viscosity subsolutions of \eqref{limHJ}}

For $i \in \cI_0$, let $\cS_i^-$ (resp., $\cS_i^+$ or $\cS_i$) be the set of
all viscosity subsolutions (resp., viscosity supersolutions or viscosity solutions ) of \eqref{limHJ}.

\begin{lem} \label{extend}
Let $i \in \cI_0$ and $u \in \mathcal S_i^-$.
Then $u$ is uniformly continuous in $J_i$, and, hence,
it can be extended uniquely to $\bar J_i$ as a continuous function on $\bar J_i$.
\end{lem}

\begin{proof}
The proof is along the same lines as that of \cite[Lemma 3.2]{K}.
By \eqref{loc-coercive}, we have
\begin{equation}
\gl u + \fr{\nu L_i}{T_i} |u'| - M \leq 0 \ \ \ \text{ in } J_i
\end{equation}
in the viscosity sense and, hence, in the almost everywhere sense.
Gronwall's inequality yields, for any $h, a \in J_i$,
\begin{equation} \label{bound-u-M}
|\gl u(h) - M| \leq |\gl u(a) - M| \exp \int_{J_i} \frac{\gl T_i (s)}{\nu L_i (s)} \, ds.
\end{equation}
Recalling that $T_i \in L^1(J_i)$, inequality \eqref{bound-u-M} shows a boundedness of $u$ in $J_i$ and, moreover,
\begin{equation} \label{unif-cont}
|u' (h)| \leq \frac{T_i (h)}{\nu L_i (h)} \left( M + \gl \sup_{J_i} |u| \right) \ \ \ \text{ for a.e. } h \in J_i
\end{equation}
yields the uniformly continuity of $u$ in $J_i$.
\end{proof}

Thanks to this lemma, we may assume that
any $u \in \cS_i^-$ is a function in $C(\bar J_i)$.
To make this notationally explicit, we write $\cS_i^- \cap C(\bar J_i)$ for $\cS_i^-$.
This also applies to $\cS_i$ since $\cS_i \sb \cS_i^-$.

The following two lemmas are direct consequences of, respectively, \eqref{unif-cont} and \eqref{bound-u-M}.

\begin{lem} \label{equi-conti-S_i^-}
Let $i \in \cI_0$ and $\cS \sb \cS_i^-$.
Assume that $\cS$ is uniformly bounded on $\bar J_i$.
Then $\cS$ is equi-continuous on $\bar J_i$.
\end{lem}

\begin{lem} \label{bound-S_i^-}
Let $i \in \cI_0$ and $u \in \cS_i^- \cap C(\bar J_i)$.
Then there exists a constant $K > 0$, independent of $u$, such that
\begin{equation*}
|u(h)| \leq K(|u(a)| + 1) \ \ \ \text{ for all } h, a \in \bar J_i.
\end{equation*} 
\end{lem}

\begin{lem} \label{u-G0}
Let $i \in \cI_0$ and $u \in \cS_i^- \cap C(\bar J_i)$. 
Then, we have $\gl u(0) + G(0, 0) \leq 0$.
\end{lem}

\begin{proof}
Fix $i \in \cI_0$.
Since $u \in \cS_i^- \cap C(\bar J_i)$,
we have
\begin{equation*}
\gl u (h) + \min_{q \in \R} \ol G_i (h, q) \leq 0 \ \ \ \text{ for all } h \in J_i.
\end{equation*}
Using Lemma \ref{limmin-G_i},
we conclude that $\gl u(0) + G(0, 0) \leq 0$.
\end{proof}

\begin{lem} \label{nu_i^d} 
Let $i \in \cI_0$ and $u \in \cS_i^- \cap C(\bar J_i)$. 
Set $d \in (-\inft, u(0))$ and
\begin{equation*} 
\nu_i^d (h) = \sup \{ v (h) \mid v \in \cS_i^- \cap C(\bar J_i), \ v(0) = d \} \ \ \ \text{ for } h \in \bar J_i.
\end{equation*}
Then there exists $\gd > 0$ such that
\begin{equation*}
\nu_i^d (h) > d \ \ \ \text{ for all } h \in J_i \cap [-\gd, \gd].
\end{equation*}
\end{lem}

The proof of this lemma is along the same lines as that of \cite[Lemma 6.8]{K}
with help of formula \eqref{G-6} and Lemma \ref{u-G0}, so we omit giving it here.
%This lemma implies that $D^+ \nu_i^d (0) \sb \R$, $\pm \infty \not\in \R$.

An important remark on $\cS_i^-$ is that
if $u \in \cS_i^-$, then $u - a \in \cS_i^-$ for any constant $a > 0$.
From this remark, the sets of all $v \in \cS_i^-$ satisfying $v(0) < u(0)$ are non-empty, and, hence,
by Lemmas \ref{equi-conti-S_i^-} and \ref{bound-S_i^-},
these are uniformly bounded and equi-continuous on $\bar J_i$.
The functions $\nu_i^d$ are thus well-defined as continuous functions on $\bar J_i$ and,
according to Perron's method, these are solutions of \eqref{limHJ}.

%%%%%%%%%%%%%%%%%%%% Proof of Theorem \ref{main} %%%%%%%%%%%%%%%%%%%%
\section{Proof of Theorem \ref{main}}

Note that the stability of viscosity solutions yields
\begin{equation*} 
-b \cd Dv^+ \leq 0 \ \ \ \text{ and } \ \ \ -b \cd Dv^- \geq 0 \ \ \ \text{ in } \gO
\end{equation*}
in the viscosity sense,
which show that $v^+$ and $v^-$ are, respectively, nondecreasing and nonincreasing
along the trajectory $\{ X(t, x) \}_{t \in \R}$, and, hence,
they are constant on the loop $c_i (h)$, $h \not= 0$.

The relations
\begin{equation*} \label{u_i^pm}
u_i^\pm (h) \in v^\pm (c_i (h))
\end{equation*}
define functions $u_i^\pm$ in $J_i$.
It is easy to check that $u_i^+$ and $u_i^-$ are, respectively,
upper and lower semicontinuous in $J_i$.

\begin{thm} \label{convergence}
For all $i \in \cI_0$, $u_i^+ \in \cS_i^-$ and $u_i^- \in \cS_i^+$.
\end{thm}

We can prove this theorem by using the same perturbed test function
as that of \cite[Theorem 3.6]{K}, so we skip the proof.

Thanks to this theorem and Lemma \ref{extend},
we may assume that $u_i^+ \in C(\bar J_i)$ for all $i \in \cI_0$.
Moreover, by (G6), we have
\begin{equation} \label{u_i^+-d_i}
u_i^+ (h_i) = \lim_{J_i \ni h \to h_i} u_i^- (h) = d_i \ \ \ \text{ for all } i \in \cI_0.
\end{equation}

In view of Theorem \ref{convergence} and \eqref{u_i^+-d_i},
to prove Theorem \ref{main}, it is enough to show that
\begin{equation} \label{enough}
v^+ (x) \leq d \leq v^- (x) \ \ \ \text{ for all } x \in c_0 (0)
\end{equation}
for some $d \in \R$.
Indeed, if \eqref{enough} is satisfied, by the semicontinuity of $v^\pm$, we have
\begin{equation*}
u_i^+ (0) \leq v^+ (x) \leq v^- (x) \leq \lim_{J_i \ni h \to 0} u_i^-(h)
\ \ \ \text{ for all } x \in c_0 (0) \text{ and } i \in \cI_0.
\end{equation*}
Thus, by the comparison principle applied to \eqref{limHJ},
we see that $u_i^+ = u_i^-$ in $J_i$ for all $i \in \cI_0$.
Hence, setting $u_i = u_i^+$ on $\bar J_i$,
we find that $u_i \in \cS_i \cap C(\bar J_i)$, $u_i (h_i) = d_i$, and $u_1 (0) = \lds = u_{N-1} (0) = d$.
Thus, we have $v^- = v^+$ on $\ol \gO \sm c_0(0)$.
Furthermore, by the definition of the half-relaxed limits and \eqref{enough}, we have
\begin{equation*}
v^- = v^+ \ \ \ \text{ on } c_0(0),
\end{equation*}
from which we conclude that, as $\ep \to 0+$,
\begin{equation*}
u^\ep \to v^+ = v^- = u_i \circ H \ \ \ \text{ on } \ol \gO_i.
\end{equation*}

The proof of \eqref{enough} can be done as an obvious combination of the two lemmas below.

\begin{lem} \label{v^--d^ast} 
Set $d = \min_{i \in \cI_0} u_i^+ (0)$. Then
\begin{equation} \label{v^-}
v^- (x) \geq d \ \ \ \text{ for all } x \in c_0 (0). 
\end{equation}
\end{lem}

\begin{lem} \label{v^+-d^ast}
Set $d = \min_{i \in \cI_0} u_i^+ (0)$. Then
\begin{equation*}
v^+ (x) \leq d \ \ \ \text{ for all } x \in c_0 (0).
\end{equation*}
\end{lem}

In order to prove these lemmas,
we need Lemmas \ref{key} and \ref{tau_2-tau_1} below.

To state Lemmas \ref{key} and \ref{tau_2-tau_1},
we set $\bar h= \min_{i \in \cI_0} |h_i|$ and,
for $h \in (0, \bar h)$, define the neighborhood $\gO (h)$ of the curve $c_0(0)$ by
\begin{equation*} \label{Omegah}
\gO_i (h) = \{ x \in \ol \gO_i \mid |H(x)| < h \} \ \text{ for } i \in \cI_0
\ \ \text{ and } \ \ \gO (h) = \bigcup_{i \in \cI_0} \gO_i (h).
\end{equation*}

\begin{lem} \label{key}
For any $\ge > 0$, there exist $\gd \in (0, \bar h)$ and $\psi \in C^1(\gO (\gd))$ such that
\begin{equation*} \label{psi}
-b \cd D\psi + G(x, 0) < G(0, 0) + \ge \ \ \ \text{ in } \gO (\gd).
\end{equation*}
\end{lem}

\begin{proof}
Fix any $\ge > 0$.
Set the function $g \in C(\ol \gO)$ by $g(x) = G(0, 0) - G(x, 0)$.
Note that, for any neighborhood $V$ of $c_0(0)$,
there is $\gd \in (0, \bar h)$ such that $\gO(\gd) \sb V$.

We fix $\bar x_i \in c_i(0) \sm \{ 0 \}$ for each $i \in \cI_1$.
Choose $r > 0$ so that $B_r \sb \ol \gO$ and
\begin{equation*}
|g(x)|<\fr{\ge}{2} \ \ \ \text{ for all } x \in B_r.
\end{equation*}
Choose $\hat g \in C^1(\gO)$ so that
\begin{equation*}
\hat g(x) = 0  \ \ \ \text{ for all } x\in B_r \ \ \ \text{ and } \ \ \ |g(x) - \hat g(x)| < \fr{\ge}{2} \ \ \ \text{ for all } x \in \gO.
\end{equation*}
Also, we choose $T > 0$ so that 
\begin{equation*}
X(t, \bar x_i) \in B_r \ \ \ \text{ for all } |t| \geq T \text{ and } i \in \cI_1,
\end{equation*}
and 
\begin{equation*}
\fr{1}{2T} \left| \int_{-T}^T \hat g(X(t ,\bar x_i)) \, dt \right| < \fr{\ge}{2} \ \ \ \text{ for all } i \in \cI_1.
\end{equation*}

For $i \in \cI_1$, let $Y_i(h)$ be the solution of the problem
\begin{equation*}
Y'(h) = \fr{DH(Y(h))}{|DH(Y(h))|^2} \ \ \ \text{ and } \ \ \ Y(0) = \bar x_i,
\end{equation*}
and note that
\begin{equation*}
Y_i \in C^1((-\bar h, \bar h); \R^2) \ \ \ \text{ and } \ \ \ H(Y_i(h)) = h \ \ \ \text{ for all } h \in (-\bar h, \bar h).
\end{equation*}

We write $Z_i(t, s)=X(t, Y_i(s))$ for $(t, s) \in \R \tim (- \bar h, \bar h)$ and $i \in \cI_1$.
For $i \in \cI_1$ and $\gd \in (-\bar h, \bar h)$, we set
\begin{equation*}
\gO_{i, T} (\gd) = \{ Z_i(t, s) \mid |t| \leq T+2, \, |s| \leq \gd \}. 
\end{equation*}
Note that the sets 
\begin{equation*}
\{ X(t, \bar x_i) \mid |t| \leq T+2 \},
\end{equation*}
with $i \in \cI_1$, are mutually disjoint,
and we choose $\gg \in (0, \bar h)$ so that 
$\gO_{i,T}(\gg)$, with $i \in \cI_1$, are mutually disjoint.

Moreover, 
we may assume that 
\begin{equation*}
Z_i(t, s) \in B_r \ \ \ \text{ for all } (t, s, i) \in ([-T-2, T+2] \sm (-T, T)) \tim [-\gg, \gg] \tim \cI_1, 
\end{equation*}
\begin{equation*}
\fr{1}{2T} \left| \int_{-T}^T \hat g(Z_i(t,s)) \, dt \right| < \fr{\ge}{2} \ \ \ \text{ for all } (s, i) \in [-\gg, \gg] \tim \cI_1,
\end{equation*}
and 
\begin{equation*}
T_i(s) > 2T+4 \ \ \ \text{ for all } (s, i) \in[-\gg, \gg] \tim \cI_0. 
\end{equation*}

For $i \in \cI_1$, set 
\begin{equation*}
\til g_i(t, s) =
\begin{cases}
\hat g(Z_i(t, s)) \ \ \ &\text{ for } (t, s) \in [-T, T] \tim [-\gg, \gg], \\
0 \ \ \ &\text{ for } (t, s) \in (\R \sm [-T, T]) \tim [-\gg, \gg],
\end{cases}
\end{equation*}
and note that $\til g_i\in C^1(\R \tim [-\gg, \gg])$.
Also, for $i \in \cI_1$, set
\begin{equation*}
m_i (s) = \fr{1}{2T} \int_{\R} \til g_i(t, s) \, dt \ \ \ \text{ for } s \in [-\gg, \gg],
\end{equation*}
and 
\begin{equation*}
\bar m_i(t, s) =
\begin{cases}
m_i (s) \ \ \ \ &\text{ for } (t, s) \in [-T, T] \tim [-\gg, \gg] \\
0 \ \ \ &\text{ for } (t,s) \in (\R \sm [-T, T]) \tim [-\gg, \gg]. 
\end{cases}
\end{equation*}
Note that 
\begin{equation*}
|\bar m_i(t, s)| < \fr{\ge}{2} \ \ \ \text{ for all } (t, s, i) \in \R \tim [-\gg, \gg] \tim \cI_1. 
\end{equation*}

Let $\gr \in C^1(\R)$ be a standard mollification kernel, with $\mathrm{ supp } \, \rho \sb (-1, 1)$.
For $i \in \cI_1$, set 
\begin{equation*}
m_i^\gr (t, s) = (\gr * \bar m_i (\cd, s))(t) \ \ \ \text{ for } (t, s) \in \R \tim [-\gg, \gg]
\end{equation*}
and 
\begin{equation*}
\til \psi_i (t, s) = -\int_{-\inft}^t(\til g_i(r, s) - m_i^\gr(r, s)) \, dr  \ \ \ \text{ for } (t, s) \in \R \tim [-\gg, \gg].  
\end{equation*}
Note that 
\begin{equation*}
m_i^\gr(t, s) = 0 \ \ \ \text{ for all } (t, s, i) \in (\R \sm (-T-1, T+1)) \tim [-\gg, \gg] \tim \cI_1,
\end{equation*}
\begin{equation*}
|m_i^\gr (t, s)| < \fr{\ge}{2} \ \ \ \text{ for all } (t, s, i) \in \R \tim [-\gg, \gg] \tim \cI_1,
\end{equation*}
and
\begin{align*}
\int_\R m_i^\gr(t, s) \, dt = \int_\R \til g_i(t, s) \, dt \ \ \ \text{ for all } (s, i) \in [-\gg, \gg] \tim \cI_1. 
\end{align*}
Note also that $\til \psi_i \in C^1(\R \tim [-\gg, \gg])$, 
\begin{equation*}
\til \psi_i (t, s) = 0 \ \ \ \text{ for all } (t, s, i) \in (\R \sm (-T-1, T+1)) \tim [-\gg, \gg] \tim \cI_1,
\end{equation*}
and 
\begin{equation*}
-\til \psi_{i,t }(t, s) = \til g_i(t, s) - m_i^\gr(t, s) < \til g_i (t, s) + \fr{\ge}{2} \ \ \ \text{ for all } (t, s, i) \in \R \tim [-\gg, \gg] \tim \cI_1.  
\end{equation*}

We show that $Z_i : [-T-2, T+2] \tim [-\gg, \gg] \to \gO_{i,T}(\gg)$ is a $C^1$ diffeomorphism. 
Obviously, by the definition of $\gO_{i,T}(\gg)$, $Z_i$ is a $C^1$ mapping and surjective.
To see that $Z_i$ is injective,
let $(t, s), (\gt, \gs) \in [-T-2, T+2] \tim [-\gg, \gg]$ be such that $Z_i(t, s) = Z_i(\gt, \gs)$. 
Note that 
\begin{equation*}
s = H(X(0,Y_i(s))) = H(Z_i (t, s)) = H(Z_i (\gt, \gs)) = \gs.
\end{equation*}
If $s=0$, we see by the standard ode theory that $t = \gt$.
If $s \not= 0$, then $t= \gt$ or $|t-\gt| \geq T_i(s)$.
The latter case is impossible since $T_i (s) > 2T+4$ and $|t - \gt| \leq 2T+4$.
Thus, we have $t = \gt$ and conclude that $Z_i$ is bijective.  
We write $(Z_i^1,Z_i^2)$ for $Z_i$ and note that 
\begin{equation} \label{1}
Z_{i,t}(t, s) = b(Z_i(t, s)) = (H_{x_2}(Z_i(t, s)) ,-H_{x_1}(Z_i(t, s))).
\end{equation}
Differentiating $H(Z_i(t, s))=s$ yields
\begin{equation*}
H_{x_1}(Z_i(t, s))Z_{i,s}^1 + H_{x_2} (Z_i(t, s))Z_{i,s}^2 = 1.
\end{equation*}
This combined with \eqref{1} reveals
\begin{equation*}
1 = H_{x_1}(Z_i(t, s))Z_{i,s}^1 + H_{x_2}(Z_i(t, s))Z_{i,s}^2 = -Z_{i,t}^2Z_{i,s}^1 + Z_{i,t}^1Z_{i,s}^2 = \det \big(Z_{i,t},  Z_{i,s}\big). 
\end{equation*}
The Inverse Function Theorem guarantees that
$Z_i : [-T-2, T+2] \tim [-\gg, \gg] \to \gO_{i,T}(\gg)$ is a $C^1$ diffeomorphism.

Note that the Inverse Function Theorem implies that
$\gO_{i,T} (\gg)$ is a neighborhood of $\{X(t, \bar x_i) \mid |t| \leq T+1\}$. Since 
\begin{equation*}
X(t, \bar x_i) \in B_r \ \ \ \text{ for all } |t| \geq T \text{ and } i \in \cI_1,
\end{equation*}
it follows that the set $B_r \cup \gO_{i,T}(\gg)$ is a neighborhood of $c_i(0)$ and, hence,
\begin{equation*}
B_r \cup \bigcup_{i \in \cI_1}\gO_{i,T} (\gg)
\end{equation*}
is a neighborhood of $c_0(0)$.
Thus, we may choose $\gd \in (0, \gg)$ so that 
\begin{equation*}
\gO(\gd) \sb B_r \cup \bigcup_{i \in \cI_1} \gO_{i,T}(\gg). 
\end{equation*}

Set  
\begin{equation*}
\psi_i (x) = \til \psi_i (Z_i^{-1}(x)) \ \ \ \text{ for } (x, i) \in \gO_{i,T}(\gg) \tim \cI_1. 
\end{equation*}
It is clear that $\psi_i \in C^1(\gO_{i,T}(\gg))$.  
We define $\psi : \gO(\gd) \to \R$ by
\begin{equation*}
\psi(x) =
\begin{cases}
\psi_i(x) \ \ \ &\text{ if } (x, i) \in \gO_{i,T}(\gg) \tim \cI_1, \\
0 \ \ \ &\text{ otherwise. }  
\end{cases}
\end{equation*}
Since $\gO_{i,T}(\gg)$, with $i \in \cI_1$, are mutually disjoint,
the function $\psi$ is well-defined.

Let $x \in \gO(\gd)$. If 
\begin{equation*}
x\in B_r \sm \bigcup_{i \in \cI_1}\gO_{i,T}(\gg), 
\end{equation*}
then $\psi=0$ in a neighborhood $V_x$
(e.g. $V_x = B_r \sm \bigcup_{i \in \cI_1}\gO_{i,T}(\gg)$) of $x$, $\psi \in C^1(V_x)$, and 
\begin{equation} \label{2}
-b(x) \cd D\psi(x) = 0 <g(x) + \ge. 
\end{equation}
Otherwise, we have 
\begin{equation*}
x \in \bigcup_{i \in \cI_1} \gO_{i,T}(\gg). 
\end{equation*}
Choose $i \in \cI_1$ so that $x \in \gO_{i,T}(\gg)$ and
set $(t, s)=Z_i^{-1}(x) \in [-T-2, T+2] \tim [-\gg, \gg]$.
Since $\gd < \gg$, we see that $|s|<\gg$. 
If $t = \pm (T+2)$, then there exists a neighborhood $V_x$ of $x$ such that
\begin{equation*}
\psi_i (x) = 0 \ \ \ \text{ for all } x \in V_x \cap \gO_{i, T}(\gg) 
\end{equation*}
and
\begin{equation*}
V_x \cap \gO_{j, T}(\gg) = \emp \ \ \ \text{ for all } j \not= i.
\end{equation*}
Since $\psi = 0$ in $V_x \sm \gO_{i, T}(\gg)$, we see that $\psi = 0$ in $V_x$, and, hence, we get \eqref{2}.
If $|t|<T+2$, then $\psi$ is of class $C^1$ in a neighborhood $V_x$ (e.g. $V_x = \gO_{i, T} (\gg)$) of $x$ and 
\begin{equation*}
\psi (y) = \til \psi_i (Z_i^{-1}(y)) \ \ \ \text{ for all } y \in V_x.
\end{equation*}
Writing $Z_i^{-1}(y) = (\gt, \gs)$ and differentiating the above, we get  
\begin{equation*}
\til \psi_{i,t} (\gt, \gs) = D\psi(y) \cd b(y)
\end{equation*}
and 
\begin{equation*}
-b(x) \cd D\psi(x) = -\psi_{i,t}(Z_i^{-1}(x)) = \til g_i(Z_i^{-1}(x)) + \fr{\ge}{2} = \hat g(x) + \fr{\ge}{2} <g(x) + \ge. 
\end{equation*}
This concludes the proof. 
\end{proof}

\begin{proof}[Proof of Lemma \ref{v^--d^ast}]
We argue by contradiction.
Thus, set $\bar d = \min_{c_0(0)} v^-$ and suppose that $\bar d < d$.
Using Lemmas \ref{nu_i^d} and \ref{key} and arguing
as in the proof of \cite[Lemma 3.8]{K}, we obtain a contradiction.
\end{proof}

Here, we note that the initial value problems
\begin{equation*}
\dot X_\pm^\ep (t) = \fr{b(X_\pm^\ep (t))}{\ep} \pm \nu \, \fr{DH(X_\pm^\ep (t))}{|DH(X_\pm^\ep (t))|}
\ \ \ \text{ and } \ \ \ X_\pm^\ep (0) = x \in \gO \sm \{ 0 \},
\end{equation*}
where $\nu > 0$ is the constant from \eqref{coercivity}, admit unique solutions $X_\pm^\ep (t, x)$
in the maximal interval $(\ul{\gs_\pm^\ep} (x), \ol{\gs_\pm^\ep} (x))$
where $\ul{\gs_\pm^\ep} (x) < 0 < \ol{\gs_\pm^\ep} (x)$, and
the maximality means that either
\begin{equation*}
\ul{\gs_\pm^\ep} (x) = -\inft \ \ \ \text{ or } \ \ \ \lim_{t \to \ul{\gs_\pm^\ep} (x) + 0} \mathrm{dist} (X_\pm^\ep (t, x), \pl \gO \cup \{ 0 \}) = 0,
\end{equation*}
and either
\begin{equation*}
\ol{\gs_\pm^\ep} (x) = \inft \ \ \ \text{ or } \ \ \ \lim_{t \to \ol{\gs_\pm^\ep} (x) - 0} \mathrm{dist} (X_\pm^\ep (t, x), \pl \gO \cup \{ 0 \}) = 0.
\end{equation*}

\begin{lem} \label{tau_2-tau_1} 
Let $\ep \in (0, \ep_0)$, $h \in (0, \bar h)$, and $x \in \gO (h) \sm \{ 0 \}$.
If $\gt_1, \gt_2 \in (\ul{\gs_+^\ep} (x), \ol{\gs_+^\ep} (x))$ are such that
$\gt_1 < \gt_2$ and $X_+^\ep (t, x) \in \gO (h)$ for all $t \in (\gt_1, \gt_2)$, then
\begin{equation*} \label{ineq-tau_2-tau_1}
\gt_2 - \gt_1 \leq \cfr{2(m+2)}{\nu c_0(m-n+2)} \, h^\fr{m-n+2}{m+2}.
\end{equation*}
Also inequality \eqref{ineq-tau_2-tau_1} holds with $\ul{\gs_+^\ep}$, $\ol{\gs_+^\ep}$, and $X_+^\ep$
being replaced by $\ul{\gs_-^\ep}$, $\ol{\gs_-^\ep}$, and $X_-^\ep$.
\end{lem}

Here, $c_0 > 0$ is the constant from \eqref{H-DH}.

\begin{proof}
We can prove this lemma by using \eqref{H-DH} and
replacing the function $\psi$ in \cite[Lemma 5.1]{K}
by $\psi (r)= r|r|^{-\frac{n}{m+2}}$ for $r \in \R \sm \{ 0 \}$.
\end{proof}

\begin{proof}[Proof of Lemma \ref{v^+-d^ast}]
We obtain \eqref{v^-} by using \eqref{solution} and Lemma \ref{tau_2-tau_1}
as well as the dynamic programming principle as in the proof of \cite[Lemma 3.7]{K}.
\end{proof}

%%%%%%%%%%%%%%% The boundary data for the limiting problem %%%%%%%%%%%%%%% 
\section{The boundary data for the limiting problem}

In this section, we present a sufficient condition on the data $d_i$ for which (G5) and (G6) hold.

Here, we only state the theorems concerning the sufficient condition and
refer to \cite[Section 7]{K} for the proofs.

For  $i \in \cI_0$,
we write $I_i$ for the set of $d \in \R$ such that the set
\begin{equation*}
\{ u \in \cS_i^- \cap C(\bar J_i) \mid u(h_i) = d \}
\end{equation*}
is nonempty.

We note that $I_i = (-\inft, a_i]$ for some $a_i \in \R$.
Indeed, in view of the remark after Lemma \ref{nu_i^d},
if $d \in I_i$ and $c < d$, then $c \in I_i$.
Also, if $d \in \R$ satisfies
\begin{equation*}
\gl d + \max_{h \in \bar J_i} \ol G_i (h, 0) \leq 0, 
\end{equation*}
then $d \in \cS_i^-$ and $d \in I_i$,
while if $d \in \R$ satisfies
\begin{equation*}
\gl d + \min_{(h, p) \in \bar J_i \tim \R} \ol G_i (h, p) > 0,
\end{equation*} 
then $d \not\in I_i$.
Thus, we see that $I_i = (-\inft, a_i]$.

For $i \in \cI_0$, $d \in I_i$, and $h \in \bar J_i$, we define
\begin{equation*}
\gr_i^d (h) = \sup \{ u(h) \mid u \in \cS_i^- \cap C(\bar J_i), \ u(h_i) = d \},
\end{equation*}
and, we have $\gr_i^d \in \cS_i\cap C(\bar J_i)$ and $\gr_i^d (h_i) = d$. 
By Lemma \ref{bound-S_i^-}, we see that
\begin{equation*}
\gr_0 := \min_{i \in \cI_0} \sup_{d \in I_i} \gr_i^d (0) < \inft.
\end{equation*}

Also, we write $I$ for the set of $d \in \R$ such that
\begin{equation*}
\{ u \in \cS_i^- \cap C(\bar J_i) \mid u(0) = d \} \not= \emptyset \ \ \ \text{ for all } i \in \cI_0.
\end{equation*}
and, for $i \in \cI_0$, $d \in I$, and $h \in \bar J_i$, we define
\begin{equation*}
\nu_i^d (h) = \sup \{ u(h) \mid u \in \cS_i^- \cap C(\bar J_i), \ u(0) = d \}.
\end{equation*}
Similarly to the above, we see that $I = (-\infty, \gr_0]$ and
that $\nu_i^d \in \cS_i \cap C(\bar J_i)$ and $\nu_i^d (0) = d$.

\begin{thm} \label{exist}
Let $(d, d_0, \lds, d_{N-1}) \in \R^{N+1}$. 
The problem
\begin{align}
\begin{cases}
\gl u_i + \ol G_i (h, u_i') = 0 \ \ \ \text{ in } J_i, \\
u_i (h_i) = d_i, \\
u_i (0) = d
\end{cases}
\end{align}
has a viscosity solution $(u_0, \lds, u_{N-1})$ $\in C(\bar J_0) \tim \lds \tim C(\bar J_{N-1})$
if and only if
\begin{align} \label{admissible}
\begin{cases}
(d, d_0, \lds, d_{N-1}) \in I \tim I_1 \tim \lds \tim I_{N-1}, \\
\min_{i \in \cI_0} \gr_i^{d_i} (0) \geq d, \\
\nu_i^{d} (h_i) \geq d_i \ \ \ \text{ for all } i \in \cI_0.
\end{cases}
\end{align} 
\end{thm}

Here, we set
\begin{equation*}
\cD = \{ (d, d_0, \lds, d_{N-1}) \in \R^{N+1} \mid \text{ \eqref{admissible} is satisfied } \},
\end{equation*}
and 
\begin{align*}
\cD_0 = \{ &(d, d_0, \lds, d_{N-1}) \in \R^{N+1} \mid \\
          &\text{ there exists $a > 0$ such that } (d + a, d_0 + a, \lds, d_{N-1} + a) \in \cD \}. 
\end{align*}

The following theorem gives a sufficient condition for which (G5) and (G6) hold.

\begin{thm}
For any $(d, d_0, \lds, d_{N-1}) \in \cD_0$,
\emph{(G5)} and \emph{(G6)} hold for some boundary data $g^\ep$. 
\end{thm}

%%%%%%%%%%%%%%%%%%%% Acknowledgments %%%%%%%%%%%%%%%%%%%%
\subsection*{Acknowledgments}

The author would like to thank Prof. Hitoshi Ishii
for many helpful comments and discussions about
the asymptotic problem for Hamilton-Jacobi equations treated here.

%%%%%%%%%%%%%%%%%%%% Funding %%%%%%%%%%%%%%%%%%%%
\subsection*{Funding}

The research has been supported by a Waseda University Grand for Special Research Projects: 2017S-070.

%%%%%%%%%%%%%%%%%%%% Reference %%%%%%%%%%%%%%%%%%%%

%%%%%%%%%%%%%%%%%%%% adress %%%%%%%%%%%%%%%%%%%%
(T. Kumagai) 
Department of Mathematics,
Faculty of Education and Integrated Arts and Sciences, Waseda University,
1-6-1 Nishi-Waseda, Shinjuku-ku, Tokyo, 169-8050, Japan

%%%%%%%%%%%%%%%%%%%% e-mail %%%%%%%%%%%%%%%%%%%%
E-mail: t.kumagai@aoni.waseda.jp

\end{document}